\documentclass[a4paper,12pt,reqno]{amsart}
\usepackage{latexsym}
\usepackage{amsmath}
\usepackage{amssymb}
\usepackage{mathrsfs}
\usepackage[all]{xy}
\usepackage{appendix}
\usepackage[pdftex]{graphicx}
\usepackage[dvips,letterpaper,margin=0.9in]{geometry}
\usepackage{braket}
\usepackage{setspace}
\usepackage{booktabs}
\usepackage{enumerate}
\usepackage{pdflscape}
\usepackage{url}
\usepackage[bottom]{footmisc}
\usepackage{tikz}

\usepackage{hyperref}
\hypersetup{
    colorlinks,
    citecolor=red,
    filecolor=black,
    linkcolor=blue,
    urlcolor=black
}

\begin{document}

\title{Conformal Holonomy of the Bivariate Gaussian Manifold}
\author{James A. Reid}
\address[James A. Reid]{Darkocean, Office 9, Building 2, Financial Square, Doha, Qatar}
\email{j.reid.06@aberdeen.ac.uk}
\date{\today}

\newcommand{\onetoone}{\overset{1:1}{\iff}}
\newcommand{\ud}{\mathrm{d}}
\newcommand{\wt}{\widetilde}
\newcommand{\wh}{\widehat}
\newcommand{\tr}{\text{tr~}}
\newcommand{\p}{\phantom}
\newcommand{\la}{\langle}
\newcommand{\ra}{\rangle}
\newcommand{\wb}{\overline}
\newcommand{\contract}{\lrcorner}
\newcommand{\dbar}{\overline{\partial}}
\newcommand{\head}[1]{\textnormal{\textbf{#1}}}
\newcommand{\poincare}{Poincar\'{e}}
\newcommand{\pe}{Poincar\'{e}-Einstein }
\newcommand{\poinein}{Poincar\'{e}-Einstein}
\newcommand{\zeroinfty}{[0, \infty)}
\newcommand{\tickbox}{\makebox[0pt][l]{$\square$}\raisebox{.15ex}{\hspace{0.1em}$\checkmark$}}
\newcommand{\longhookrightarrow}{\lhook\joinrel\longrightarrow}

\newtheorem{definition}{Definition}
\newtheorem{lemma}{Lemma}
\newtheorem{theorem}{Theorem}
\newtheorem{corollary}{Corollary}
\newtheorem{calculation}{Calculation}
\newtheorem{proposition}{Proposition}
\newtheorem{fact}{Fact}
\newtheorem{mainresult}{Main Result}
\newtheorem{identity}{Identity}
\newtheorem{speculation}{Speculation}

\newcommand{\cone}{\mbox{\textcircled{\small 1}}}
\newcommand{\ctwo}{\mbox{\textcircled{\small 2}}}
\newcommand{\cthree}{\mbox{\textcircled{\small 3}}}
\newcommand{\cfour}{\mbox{\textcircled{\small 4}}}
\newcommand{\cfive}{\mbox{\textcircled{\small 5}}}

\newcommand{\circex}{\mbox{\textcircled{\small \boldsymbol{!}}}}
\newcommand{\circa}{\mbox{\textcircled{\small a}}}
\newcommand{\circb}{\mbox{\textcircled{\small b}}}
\newcommand{\circc}{\mbox{\textcircled{\small c}}}
\newcommand{\circd}{\mbox{\textcircled{\small d}}}

\begin{abstract}
Statistical manifolds, the parameter spaces of smooth families of probability density functions, are the central objects of study in information geometry.  While the elementary differential geometry of Riemannian statistical manifolds is well-known, their conformal geometry remains entirely unexplored.  In this article, we begin this programme of exploration by determining some invariants of the conformal structure of the Fisher-Rao metric.  Specifically, we study the holonomy of a conformally-invariant connection on the standard tractor bundle of the bivariate Gaussian manifold.  It is found that for a generic pair of random variables, the conformal holonomy group is the identity-connected component of the indefinite special orthogonal group, $SO^{0}(1,6)$.  Remarkably, however, when the random variables are statistically independent, the conformal holonomy representation is reducible and the conformal holonomy group is $SO^{0}(1,4)$.\\ \\   
{\it Mathematics Subject Classification 2020}: 53B12, 53C18, 53C25, 53C29 and 62H10.\\
{\it Keywords}: Normal distribution manifolds, conformal holonomy, Einstein metrics.
\end{abstract}
\maketitle
%
%

\section{Information geometry}
\subsection{Preliminaries}
To preface our exposition, we recall some facts about the differential geometry of statistical manifolds, with a focus on the bivariate Gaussian manifold.  Our review follows \cite{arwini} and \cite{li}, but is not exhaustive.  Let    $x$ be a random variable in a parameter space $\Theta$.  A parametric statistical model $S$ is an $n$-dimensional smooth family $\left\{ p_{\theta}~|~\theta\in\Theta\right\}$ of probability density functions $p\left(x;\theta\right)$ on an event space $\Omega$ of unit measure. The $n$-tuple $\theta=\left(\theta_{1},\theta_{2},\dots,\theta_{n}\right)\in\Theta \subset \mathbb{R}^{n}$ furnishes a set of natural coordinates on $S$, which is endowed with a differentiable structure, so we equivalently say that $S$ is a statistical manifold. Continuing, we say that an $n$-dimensional parametric statistical model $\left\{ p_{\theta}~|~\theta\in\Theta\right\}$ is an exponential family if the probability density function can be expressed in terms of functions $C,F_{1},\dots,F_{n}$ and a convex function $\varphi$ on $\Theta$ of the form
\begin{equation}
p(x; \theta) = \exp \left\{ C(x) + \Sigma_{a}\theta_{a} F_{a}(x) - \varphi(\theta) \right\}.
\end{equation}
Here, $\left\{ \theta_{a}\right\}$ are called the natural parameters and $\varphi$ is the potential function.  Given a statistical manifold $S$ with log-likelihood function $l:=l(x;\theta)=\log p(x;\theta)$, and denoting differentiation with respect to the natural parameters by $\partial_{a} := \partial / \partial\theta_{a}$, one may define a Riemannian metric on $S$ as
\begin{equation}\label{eq:fisherrao}
g_{ab}(\theta):= \mathbb{E}[(\partial_{a}l)(\partial_{b}l)]=\int \ud x~ p(x;\theta) (\partial_{a}l)(\partial_{b}l),
\end{equation}
where $a,b=1,\ldots,n$, and $\mathbb{E}$ denotes the expectation value.  Equation \eqref{eq:fisherrao} is the Fisher-Rao information metric which on exponential families reduces to the form $g_{ab}=\partial_a\partial_b\varphi$. Adopting the Einstein summation convention on repeated covariant and contravariant indices, the Christoffel symbol of the Levi-\v{C}ivit\`{a} connection $\nabla$ associated to the Fisher-Rao metric $g$ are given by
\begin{equation*}
\Gamma^{c}_{\p{c}ab} := \frac{1}{2} g^{cd} \left( \partial_{a} g_{bd} + \partial_{b} g_{ad} - \partial_{d} g_{ab} \right) = \varphi^{cd}(\theta)\varphi_{abd}(\theta).
\end{equation*}
We now consider one of the central motifs of this article, the bivariate Gaussian manifold.  Following \cite{li}. let $\text{PD}(d,\mathbb{R})$ be the set of all real $d$-ordered positive definite symmetric matrices.  The $d$-dimensional normal distribution manifold is defined by:
\begin{equation}\label{eq:defg}
N^{d} := \left\{ p(x,\mu,\Sigma) = \frac{\exp \left( -\frac{1}{2} (x - \mu)^{\text{T}} \Sigma^{-1} (x - \mu) \right)}{(2 \pi)^{\frac{d}{2}} \left( \det \Sigma \right)^{\frac{1}{2}}} ~\Bigg|~ \mu \in \mathbb{R}^{d},~ \Sigma \in \text{PD}(d,\mathbb{R}) \right\}.
\end{equation}
Consequently, with respect to the parameters $(\mu_{1},\mu_{2},\sigma_{1},\sigma_{2},\sigma_{12})$, where $\mu_{1},\mu_{2} \in \mathbb{R}$, $\sigma_{1},\sigma_{2} \in \mathbb{R}^{+}$ with covariance $\sigma_{12}$, the bivariate Gaussian manifold is defined as
\begin{align}
\mathcal{G} = \Bigg\{ p(x,&y,\mu_{1},\mu_{2},\sigma_{1},\sigma_{2},\sigma_{12}) = \nonumber \\ 
& \frac{1}{2 \pi \sqrt{\sigma_1 \sigma_2 - \sigma_{12}^{\p{12}2}}} \exp \left( \frac{\sigma_{2}(x - \mu_{1})^{2} - 2 \sigma_{12}(x - \mu_{1})(y - \mu_{2}) + \sigma_{1}(y - \mu_{2})^{2}}{2 \left( \sigma_1 \sigma_2 - \sigma_{12}^{\p{12}2} \right)} \right) \Bigg\}.
\end{align}
It is clear by inspection that $\mathcal{G}$ has five real dimensions. The set of all bivariate Gaussian density functions that constitute $\mathcal{G}$ form an exponential family with natural coordinate system
\begin{equation}\label{eq:naturalcoords}
(\theta_{1},\theta_{2},\theta_{3},\theta_{4},\theta_{5}) = \left( \frac{\mu_{1}\sigma_{2} - \mu_{2}\sigma_{12}}{\Delta}, \frac{\mu_{2}\sigma_{1} - \mu_{1}\sigma_{12}}{\Delta}, -\frac{\sigma_{2}}{2\Delta}, \frac{\sigma_{12}}{\Delta}, -\frac{\sigma_{1}}{2\Delta} \right)
\end{equation}
with corresponding potential function
\begin{align*}
\varphi(\theta) &= \log \left( 2 \pi \sqrt{\Delta} \right) + \frac{\mu_{2}^{\p{2}2}\sigma_{1} + \mu_{1}^{\p{1}2}\sigma_{2} - 2\mu_{1}\mu_{2}\sigma_{12}}{2\Delta} \\
 &= \log \left( 2 \pi \sqrt{\Delta} \right) - \Delta \left( \theta_{2}^{\p{2}2}\theta_{3} - \theta_{1}\theta_{2}\theta_{4} + \theta_{1}^{\p{1}2}\theta_{5} \right),
\end{align*}
where
\begin{equation*}
\Delta = \sigma_{1}\sigma_{2} - \sigma_{12}^{\p{12}2} = \frac{1}{4\theta_{3}\theta_{5} - \theta_{4}^{\p{4}2}}.
\end{equation*}
The Fisher-Rao metric on $\mathcal{G}$ takes the form
\begin{equation}
g_{ab} = \left[
\begin{array}{ccccc}
\frac{\sigma_{2}}{\Delta} & -\frac{\sigma_{12}}{\Delta} & 0 & 0 & 0 \\
-\frac{\sigma_{12}}{\Delta} & \frac{\sigma_{1}}{\Delta} & 0 & 0 & 0 \\
0 & 0 & \frac{\sigma_{2}^{\p{2}2}}{2\Delta^{2}} & -\frac{\sigma_{12}\sigma_{2}}{\Delta^{2}} & \frac{\sigma_{12}^{\p{11}2}}{2\Delta^{2}} \\
0 & 0 & -\frac{\sigma_{12}\sigma_{2}}{\Delta^{2}} & \frac{\sigma_{1}\sigma_{2} + \sigma_{12}^{\p{12}2}}{\Delta^{2}} & -\frac{\sigma_{1}\sigma_{12}}{\Delta^{2}} \\
0 & 0 & \frac{\sigma_{12}^{\p{11}2}}{2\Delta^{2}} & -\frac{\sigma_{1}\sigma_{12}}{\Delta^{2}} & \frac{\sigma_{1}^{\p{1}2}}{2\Delta^{2}} \\
\end{array}{}
\right].
\end{equation}
Statistical manifolds support a one-parameter family of affine connections called $\alpha$-connections.  The connection components of the $\alpha$-connection with respect to the local coordinates $(\mu_{1},\mu_{2},\sigma_{1},\sigma_{2},\sigma_{12})$ of the bivariate Gaussian manifold are given by:
\begin{equation}
\Gamma^{(\alpha)}_{ab,c}(\xi) = \int_{-\infty}^{\infty} \int_{-\infty}^{\infty} \ud x~ \ud y \left( \frac{\partial^{2}f}{\partial \xi^{a} \partial \xi^{b}} \frac{\partial \log f}{\partial \xi^{c}} + \frac{1-\alpha}{2} \frac{\partial \log f}{\partial \xi^{a}} \frac{\partial \log f}{\partial \xi^{b}} \frac{\partial \log f}{\partial \xi^{c}} \right) f,
\end{equation}
or equivalently, with respect to the natural coordinates $(\theta_{i})$, are:
\begin{equation}
\Gamma^{(\alpha)}_{ab,c}(\theta) = \frac{1-\alpha}{2} \partial_{a} \partial_{b} \partial_{c} \varphi(\theta).
\end{equation}
The only $\alpha$-connection that is compatible with the Fisher-Rao metric is $\nabla = \nabla^{(0)}$. $\nabla$ is torsion-free and so coincides with the usual Levi-\v{C}ivit\`{a} connection.  The non-vanishing components of $\nabla$ are:
\begin{align}
&\Gamma_{ab}^{\p{ij}1} = \left[
\begin{array}{ccccc}
0 & 0 & -\frac{\sigma_{2}}{2\Delta} & \frac{\sigma_{12}}{2\Delta} & 0 \\
0 & 0 & \frac{\sigma_{12}}{2\Delta} & -\frac{\sigma_{1}}{2\Delta} & 0 \\
-\frac{\sigma_{2}}{2\Delta} & \frac{\sigma_{12}}{2\Delta} & 0 & 0 & 0 \\
\frac{\sigma_{12}}{2\Delta} & \frac{\sigma_{1}}{2\Delta} & 0 & 0 & 0 \\
0 & 0 & 0 & 0 & 0
\end{array}{}
\right],~~~
\Gamma_{ab}^{\p{ij}2} = \left[
\begin{array}{ccccc}
0 & 0 & 0 & -\frac{\sigma_{2}}{2\Delta} & \frac{\sigma_{12}}{2\Delta} \\
0 & 0 & 0 & \frac{\sigma_{12}}{2\Delta} & -\frac{\sigma_{1}}{2\Delta} \\
0 & 0 & 0 & 0 & 0 \\
-\frac{\sigma_{2}}{2\Delta} & \frac{\sigma_{12}}{2\Delta} & 0 & 0 & 0 \\
\frac{\sigma_{12}}{2\Delta} & \frac{\sigma_{1}}{2\Delta} & 0 & 0 & 0
\end{array}{}
\right],\nonumber \\
&\Gamma_{ab}^{\p{ij}3} = \left[
\begin{array}{ccccc}
1 & 0 & 0 & 0 & 0 \\
0 & 0 & 0 & 0 & 0 \\
0 & 0 & -\frac{\sigma_{2}}{\Delta} & \frac{\sigma_{12}}{\Delta} & 0 \\
0 & 0 & \frac{\sigma_{12}}{\Delta} & -\frac{\sigma_{1}}{\Delta} & 0 \\
0 & 0 & 0 & 0 & 0 \\
\end{array}{}
\right],~~~
\Gamma_{ab}^{\p{ij}4} = \left[
\begin{array}{ccccc}
0 & \frac{1}{2} & 0 & 0 & 0 \\
\frac{1}{2} & 0 & 0 & 0 & 0 \\
0 & 0 & 0 & -\frac{\sigma_{2}}{2\Delta} & \frac{\sigma_{12}}{2\Delta} \\
0 & 0 & -\frac{\sigma_{2}}{2\Delta} & \frac{\sigma_{12}}{\Delta} & -\frac{\sigma_{1}}{2\Delta} \\
0 & 0 & \frac{\sigma_{12}}{2\Delta} & -\frac{\sigma_{1}}{2\Delta} & 0 \\
\end{array}{}
\right], \nonumber \\ 
&\Gamma_{ab}^{\p{ij}5} = \left[
\begin{array}{ccccc}
0 & 0 & 0 & 0 & 0 \\
0 & 1 & 0 & 0 & 0 \\
0 & 0 & 0 & 0 & 0 \\
0 & 0 & 0 & -\frac{\sigma_{2}}{2\Delta} & \frac{\sigma_{12}}{2\Delta} \\
0 & 0 & 0 & \frac{\sigma_{12}}{2\Delta} & -\frac{\sigma_{1}}{2\Delta} \\
\end{array}{}
\right].
\end{align}
The curvature tensors of the Fischer-Rao metric are used in calculations in our proofs in section \ref{section:confholg}, so we present them below despite their long forms.  The non-vanishing independent components $\mathsf{R}_{abcd}$ of the Riemann tensor are:
\begin{align}
&\mathsf{R}_{12cd} = - \left[
\begin{array}{ccccc}
0 & \frac{1}{4\Delta} & 0 & 0 & 0 \\
-\frac{1}{4\Delta} & 0 & 0 & 0 & 0 \\
0 & 0 & 0 & -\frac{\sigma_{2}}{4\Delta^{2}} & \frac{\sigma_{12}}{4\Delta^{2}} \\
0 & 0 & \frac{\sigma_{2}}{4\Delta^{2}} & 0 & -\frac{\sigma_{1}}{4\Delta^{2}} \\
0 & 0 & -\frac{\sigma_{12}}{4\Delta^{2}} & \frac{\sigma_{1}}{4\Delta^{2}} & 0 \\
\end{array}{}
\right], \nonumber \\
&\mathsf{R}_{13cd} = - \left[
\begin{array}{ccccc}
0 & 0 & -\frac{\sigma_{2}^{\p{2}3}}{4\Delta^{3}} & \frac{\sigma_{2}^{\p{2}2}\sigma_{12}}{2\Delta^{3}} & -\frac{\sigma_{2}\sigma_{12}^{\p{12}2}}{4\Delta^{3}} \\
0 & 0 & \frac{\sigma_{2}^{\p{2}2}\sigma_{12}}{4\Delta^{3}} & -\frac{\sigma_{2}(\sigma_{1}\sigma_{2}+\sigma_{12}^{\p{12}2})}{4\Delta^{3}} & \frac{\sigma_{1}\sigma_{2}\sigma_{12}}{4\Delta^{3}} \\
\frac{\sigma_{2}^{\p{2}2}}{4\Delta^{3}} & -\frac{\sigma_{2}^{\p{2}2}\sigma_{12}}{4\Delta^{3}} & 0 & 0 & 0 \\
-\frac{\sigma_{2}^{\p{2}2}\sigma_{12}}{2\Delta^{3}} & \frac{\sigma_{2}(\sigma_{1}\sigma_{2}+\sigma_{12}^{\p{12}2})}{4\Delta^{3}} & 0 & 0 & 0 \\
\frac{\sigma_{2}\sigma_{12}^{\p{12}2}}{4\Delta^{3}} & -\frac{\sigma_{1}\sigma_{2}\sigma_{12}}{4\Delta^{3}} & 0 & 0 & 0 \\
\end{array}{}
\right], \nonumber \\
&\mathsf{R}_{14cd} = - \left[
\begin{array}{ccccc}
0 & 0 & \frac{\sigma_{2}^{\p{2}2}\sigma_{12}}{2\Delta^{3}} & -\frac{\sigma_{2}(\sigma_{1}\sigma_{2}+3\sigma_{12}^{\p{12}2})}{4\Delta^{3}} & \frac{\sigma_{12}(\sigma_{1}\sigma_{2}+\sigma_{12}^{\p{12}2})}{4\Delta^{3}} \\
0 & 0 & -\frac{\sigma_{2}\sigma_{12}}{2\Delta^{3}} & \frac{\sigma_{12}(3\sigma_{1}\sigma_{2}+\sigma_{12}^{\p{12}2})}{4\Delta^{3}} & -\frac{\sigma_{1}(\sigma_{1}\sigma_{2}+\sigma_{12}^{\p{12}2})}{4\Delta^{3}} \\
-\frac{\sigma_{2}^{\p{2}2}\sigma_{12}}{2\Delta^{3}} & \frac{\sigma_{2}\sigma_{12}}{2\Delta^{3}} & 0 & 0 & 0 \\
\frac{\sigma_{2}(\sigma_{1}\sigma_{2}+3\sigma_{12}^{\p{12}2})}{4\Delta^{3}} & -\frac{\sigma_{12}(3\sigma_{1}\sigma_{2}+\sigma_{12}^{\p{12}2})}{4\Delta^{3}} & 0 & 0 & 0 \\
-\frac{\sigma_{12}(\sigma_{1}\sigma_{2}+\sigma_{12}^{\p{12}2})}{4\Delta^{3}} & \frac{\sigma_{1}(\sigma_{1}\sigma_{2}+\sigma_{12}^{\p{12}2})}{4\Delta^{3}}  & 0 & 0 & 0 \\
\end{array}{}
\right], \nonumber \\
&\mathsf{R}_{15cd} = - \left[
\begin{array}{ccccc}
0 & 0 & -\frac{\sigma_{2}^{\p{2}2}\sigma_{12}^{\p{12}2}}{4\Delta^{3}} & \frac{\sigma_{12}(\sigma_{1}\sigma_{2}+\sigma_{12}^{\p{12}2})}{4\Delta^{3}} & -\frac{\sigma_{1}\sigma_{12}^{\p{12}2}}{4\Delta^{3}} \\
0 & 0 & \frac{\sigma_{12}^{\p{12}3}}{4\Delta^{3}} & -\frac{\sigma_{1}\sigma_{12}^{\p{12}2}}{2\Delta^{3}} & \frac{\sigma_{1}^{\p{1}2}\sigma_{12}}{4\Delta^{3}} \\
\frac{\sigma_{2}^{\p{2}2}\sigma_{12}^{\p{12}2}}{4\Delta^{3}} & -\frac{\sigma_{12}^{\p{12}3}}{4\Delta^{3}} & 0 & 0 & 0 \\
-\frac{\sigma_{12}(\sigma_{1}\sigma_{2}+\sigma_{12}^{\p{12}2})}{4\Delta^{3}} & \frac{\sigma_{1}\sigma_{12}^{\p{12}2}}{2\Delta^{3}} & 0 & 0 & 0 \\
\frac{\sigma_{1}\sigma_{12}^{\p{12}2}}{4\Delta^{3}} & -\frac{\sigma_{1}^{\p{1}2}\sigma_{12}}{4\Delta^{3}} & 0 & 0 & 0 \\
\end{array}{}
\right], \nonumber
\end{align}
\begin{align}\label{eq:fr-riemann}
&\mathsf{R}_{23cd} = - \left[
\begin{array}{ccccc}
0 & 0 & \frac{\sigma_{2}^{\p{2}2}\sigma_{12}}{4\Delta^{3}} & -\frac{\sigma_{2}\sigma_{12}^{\p{12}2}}{2\Delta^{3}} & \frac{\sigma_{12}^{\p{12}3}}{4\Delta^{3}} \\
0 & 0 & -\frac{\sigma_{2}\sigma_{12}^{\p{12}2}}{4\Delta^{3}} & \frac{\sigma_{12}(\sigma_{1}\sigma_{2}+\sigma_{12}^{\p{12}2})}{4\Delta^{3}} & -\frac{\sigma_{1}\sigma_{12}^{\p{12}2}}{4\Delta^{3}} \\
-\frac{\sigma_{2}^{\p{2}2}\sigma_{12}}{4\Delta^{3}} & \frac{\sigma_{2}\sigma_{12}^{\p{12}2}}{4\Delta^{3}} & 0 & 0 & 0 \\
\frac{\sigma_{2}\sigma_{12}^{\p{12}2}}{2\Delta^{3}} & -\frac{\sigma_{12}(\sigma_{1}\sigma_{2}+\sigma_{12}^{\p{12}2})}{4\Delta^{3}} & 0 & 0 & 0 \\
-\frac{\sigma_{12}^{\p{12}3}}{4\Delta^{3}} & \frac{\sigma_{1}\sigma_{12}^{\p{12}2}}{4\Delta^{3}} & 0 & 0 & 0 \\
\end{array}{}
\right], \nonumber \\
&\mathsf{R}_{24cd} = - \left[
\begin{array}{ccccc}
0 & 0 & -\frac{\sigma_{2}(\sigma_{1}\sigma_{2}+\sigma_{12}^{\p{12}2})}{4\Delta^{3}} & \frac{\sigma_{12}(3\sigma_{1}\sigma_{2}+\sigma_{12}^{\p{12}2})}{4\Delta^{3}} & -\frac{\sigma_{1}\sigma_{12}^{\p{12}2}}{2\Delta^{3}} \\
0 & 0 & \frac{\sigma_{12}(\sigma_{1}\sigma_{2}+\sigma_{12}^{\p{12}2})}{4\Delta^{3}} & -\frac{\sigma_{1}(\sigma_{1}\sigma_{2}+3\sigma_{12}^{\p{12}2})}{4\Delta^{3}} & \frac{\sigma_{1}^{\p{1}2}\sigma_{12}}{2\Delta^{3}} \\
\frac{\sigma_{2}(\sigma_{1}\sigma_{2}+\sigma_{12}^{\p{12}2})}{4\Delta^{3}} & -\frac{\sigma_{12}(\sigma_{1}\sigma_{2}+\sigma_{12}^{\p{12}2})}{4\Delta^{3}} & 0 & 0 & 0 \\
-\frac{\sigma_{12}(3\sigma_{1}\sigma_{2}+\sigma_{12}^{\p{12}2})}{4\Delta^{3}} & \frac{\sigma_{1}(\sigma_{1}\sigma_{2}+3\sigma_{12}^{\p{12}2})}{4\Delta^{3}} & 0 & 0 & 0 \\
\frac{\sigma_{1}\sigma_{12}^{\p{12}2}}{2\Delta^{3}} & -\frac{\sigma_{1}^{\p{1}2}\sigma_{12}}{2\Delta^{3}}  & 0 & 0 & 0 \\
\end{array}{}
\right], \nonumber \\
&\mathsf{R}_{25cd} = - \left[
\begin{array}{ccccc}
0 & 0 & \frac{\sigma_{1}\sigma_{2}\sigma_{12}}{4\Delta^{3}} & -\frac{\sigma_{1}(\sigma_{1}\sigma_{2}+\sigma_{12}^{\p{12}2})}{4\Delta^{3}} & \frac{\sigma_{1}^{\p{1}2}\sigma_{12}}{4\Delta^{3}} \\
0 & 0 & -\frac{\sigma_{1}\sigma_{12}^{\p{12}2}}{4\Delta^{3}} & \frac{\sigma_{1}^{\p{1}2}\sigma_{12}}{2\Delta^{3}} & -\frac{\sigma_{1}^{\p{1}3}}{4\Delta^{3}} \\
-\frac{\sigma_{1}\sigma_{2}\sigma_{12}}{4\Delta^{3}} & \frac{\sigma_{1}\sigma_{12}^{\p{12}2}}{4\Delta^{3}} & 0 & 0 & 0 \\
\frac{\sigma_{1}(\sigma_{1}\sigma_{2}+\sigma_{12}^{\p{12}2})}{4\Delta^{3}} & -\frac{\sigma_{1}^{\p{1}2}\sigma_{12}}{2\Delta^{3}} & 0 & 0 & 0 \\
-\frac{\sigma_{1}^{\p{1}2}\sigma_{12}}{4\Delta^{3}} & \frac{\sigma_{1}^{\p{1}3}}{4\Delta^{3}} & 0 & 0 & 0 \\
\end{array}{}
\right], \nonumber \\
&\mathsf{R}_{34cd} = - \left[
\begin{array}{ccccc}
0 & \frac{\sigma_{2}}{4\Delta^{2}} & 0 & 0 & 0 \\
-\frac{\sigma_{2}}{4\Delta^{2}} & 0 & 0 & 0 & 0 \\
0 & 0 & 0 & -\frac{\sigma_{2}^{\p{2}2}}{4\Delta^{3}} & \frac{\sigma_{2}\sigma_{12}}{4\Delta^{3}} \\
0 & 0 & \frac{\sigma_{2}^{\p{2}2}}{4\Delta^{3}} & 0 & -\frac{\sigma_{1}\sigma_{2}}{4\Delta^{3}} \\
0 & 0 & -\frac{\sigma_{2}\sigma_{12}}{4\Delta^{3}}  & \frac{\sigma_{1}\sigma_{2}}{4\Delta^{3}} & 0 \\
\end{array}{}
\right], \nonumber \\
&\mathsf{R}_{35cd} = - \left[
\begin{array}{ccccc}
0 & \frac{\sigma_{12}}{4\Delta^{2}} & 0 & 0 & 0 \\
-\frac{\sigma_{12}}{4\Delta^{2}} & 0 & 0 & 0 & 0 \\
0 & 0 & 0 & \frac{\sigma_{2}\sigma_{12}}{4\Delta^{3}} & -\frac{\sigma_{12}^{\p{12}2}}{4\Delta^{3}} \\
0 & 0 & -\frac{\sigma_{2}\sigma_{12}}{4\Delta^{3}} & 0 & \frac{\sigma_{1}\sigma_{12}}{4\Delta^{3}} \\
0 & 0 & \frac{\sigma_{12}^{\p{12}2}}{4\Delta^{3}} & -\frac{\sigma_{1}\sigma_{12}}{4\Delta^{3}} & 0 \\
\end{array}{}
\right], \nonumber \\
&\mathsf{R}_{45cd} = - \left[
\begin{array}{ccccc}
0 & -\frac{\sigma_{1}}{4\Delta^{2}}  & 0 & 0 & 0 \\
\frac{\sigma_{1}}{4\Delta^{2}} & 0 & 0 & 0 & 0 \\
0 & 0 & 0 & -\frac{\sigma_{1}\sigma_{2}}{4\Delta^{3}} & \frac{\sigma_{1}\sigma_{12}}{4\Delta^{3}} \\
0 & 0 & \frac{\sigma_{1}\sigma_{2}}{4\Delta^{3}} & 0 & -\frac{\sigma_{1}^{\p{1}2}}{4\Delta^{3}} \\
0 & 0 & -\frac{\sigma_{1}\sigma_{12}}{4\Delta^{3}} & \frac{\sigma_{1}^{\p{1}2}}{4\Delta^{3}} & 0 \\
\end{array}{}
\right]
\end{align}
The Ricci tensor $\mathsf{Ric}_{ab}$ is given by:
\begin{equation}\label{eq:fr-ricci}
\mathsf{Ric}_{ab} = - \left[
\begin{array}{ccccc}
\frac{\sigma_{2}}{2\Delta}  & -\frac{\sigma_{12}}{2\Delta} & 0 & 0 & 0 \\
-\frac{\sigma_{12}}{2\Delta} & \frac{\sigma_{1}}{2\Delta} & 0 & 0 & 0 \\
0 & 0 & \frac{\sigma_{2}^{\p{2}2}}{2\Delta^{2}}  & -\frac{\sigma_{2}\sigma_{12}}{\Delta^{2}} & \frac{3\sigma_{12}^{\p{12}2}-\sigma_{1}\sigma_{2}}{4\Delta^{2}} \\
0 & 0 & -\frac{\sigma_{2}\sigma_{12}}{\Delta^{2}} & \frac{3\sigma_{1}\sigma_{2} + \sigma_{12}^{\p{12}2}}{2\Delta^{2}} & -\frac{\sigma_{1}\sigma_{12}}{\Delta^{2}} \\
0 & 0 & \frac{3\sigma_{12}^{\p{12}2}-\sigma_{1}\sigma_{2}}{4\Delta^{2}} & -\frac{\sigma_{1}\sigma_{12}}{\Delta^{2}} & \frac{\sigma_{1}^{\p{1}2}}{2\Delta^{2}}  \\
\end{array}{}
\right],
\end{equation}
and the scalar curvature $\mathsf{scal}$ is given by:
\begin{equation}
\mathsf{scal} = - \frac{9}{2}.
\end{equation}
Lastly, the sectional curvature of the Fischer-Rao metric is:
\begin{equation}
\mathsf{k} = - \left[
\begin{array}{ccccc}
0 & -\frac{1}{4} & \frac{1}{2} & \frac{\sigma_{1}\sigma_{2}+3\sigma_{12}^{\p{12}2}}{4(\sigma_{1}\sigma_{2}+\sigma_{12}^{\p{12}2})} & \frac{\sigma_{12}^{\p{12}2}}{2\sigma_{1}\sigma_{2}} \\
-\frac{1}{4} & 0 & \frac{\sigma_{12}^{\p{12}2}}{2\sigma_{1}\sigma_{2}} & \frac{\sigma_{1}\sigma_{2}+3\sigma_{12}^{\p{12}2}}{4(\sigma_{1}\sigma_{2}+\sigma_{12}^{\p{12}2})} & \frac{1}{2} \\
\frac{1}{2} & \frac{\sigma_{12}^{\p{12}2}}{2\sigma_{1}\sigma_{2}} & 0 & \frac{1}{2} & \frac{\sigma_{12}^{\p{12}2}}{\sigma_{1}\sigma_{2}+\sigma_{12}^{\p{12}2}} \\
\frac{\sigma_{1}\sigma_{2}+3\sigma_{12}^{\p{12}2}}{4(\sigma_{1}\sigma_{2}+\sigma_{12}^{\p{12}2})} & \frac{\sigma_{1}\sigma_{2}+3\sigma_{12}^{\p{12}2}}{4(\sigma_{1}\sigma_{2}+\sigma_{12}^{\p{12}2})} & \frac{1}{2} & 0 & \frac{1}{2} \\
\frac{\sigma_{12}^{\p{12}2}}{2\sigma_{1}\sigma_{2}} & \frac{1}{2} & \frac{\sigma_{12}^{\p{12}2}}{\sigma_{1}\sigma_{2}+\sigma_{12}^{\p{12}2}} & \frac{1}{2} & 0 \\
\end{array}{}
\right].
\end{equation}


\section{Conformal geometry}
\subsection{Conformal structures and tractor bundles}
To preface our exposition, we recall the modern approach to conformal geometry as discussed in \cite{governurowski, baumreview} and \cite{eastwood}.  Other excellent reviews are given in \cite{bailey, goverholography, govershaukatwaldron2, govershaukatwaldron1, goverboundary} and \cite{leitnerhab}.  Consider an $n$-dimensional manifold $M$ endowed with a semi-Riemannian metric $g$ of signature $(p,q)$.  The conformal class $[g]$ of the metric $g$ is the set of metrics
\begin{equation}
[g] :=  \Big\{ e^{2 \Upsilon} g ~\Big|~ \Upsilon \in C^{\infty}(M) \Big\}
\end{equation}
which are equivalent to $g$ upto a scaling by a positive function, $e^{2 \Upsilon}$.  We denote the conformal class of a semi-Riemannian manifold by $c$ when we do not wish to privilege a particular metric.  When there exists an $\Upsilon$ such that $\wh{g} = e^{2 \Upsilon} g$, the metrics $\wh{g}$ and $g$ are said to be conformally related (or conformally equivalent) and the map $g \to \wh{g}$ is called a conformal transformation.  The pair $(M,[g])$ is called a conformal manifold and $[g]$ a conformal structure for $M$.  Importantly, when the conformal class $[g]$ on a manifold $M$ contains an Einstein metric, $(M,[g])$ is said to be conformally Einstein.  Lastly, we recall Penrose's abstract index notation, where a single kernel letter ($\mathcal{E}$ is the usual choice) is adorned with non-numerical \emph{abstract} indices according to the bundle structure.  For example, a vector field $V$ would be written as $V^{a} \in \Gamma(\mathcal{E}^{a})$, where $\mathcal{E}^{a}$ is the abstractly indexed tangent bundle, and so on.\\ \\
Returning to our discussion of conformal structures, given two metrics $\wh{g}$ and $g$ in the conformal class and a point $x \in M$, there exists some positive number $s \in \mathbb{R}^{+}$ such that the metrics are multiples of eachother:
\begin{equation}\label{eq:realmultiplication}
\wh{g}_{x} = s g_{x}.
\end{equation}
The conformal structure can be described as a smooth ray sub-bundle $\mathcal{Q}$ of the bundle of metrics, $\odot^{2} T^{*}M$.  The fibre $\mathcal{Q}_{x}$ consists of conformally related metrics at the point $x$, and so points of $\mathcal{Q}$ are pairs of the form $(g_{x}, x)$.  By construction, sections of $\mathcal{Q}$ are in bijective correspondence with metrics in the conformal class. $\mathcal{Q}$ is an $\mathbb{R}^{+}$-principal bundle on $M$, and using the irreducible representation
\begin{equation}
\mathbb{R}^{+} \ni x \longmapsto x^{-w/2} \in \text{End}(\mathbb{R}^{+}),
\end{equation}
one may construct for each $w$ an associated line bundle $\mathcal{E}[w]$ over $(M,[g])$ called a conformal density bundle\footnote{$\mathcal{E}[w]$ are called density bundles because they may be identified with powers of the 1-density bundle (associated to the frame bundle) via the representation $|\text{det}(\cdot)|^{-1}$.}, which is trivialised by a choice of metric $g \in c$.  Sections of $\mathcal{E}[w]$ are called conformal densities of weight $w$, where $w$ is the conformal weight.  One may construct conformally weighted quantities by twisting an appropriate bundle with a conformal density bundle.  For example, one may write a conformally weighted covector $\eta$ of conformal weight $w$ as $\eta_{a} \in \Gamma(\mathcal{E}_{a}[w])$, where $\mathcal{E}_{a}[w] := \mathcal{E}_{a} \otimes \mathcal{E}[w]$.  Importantly, for $w=1$, sections $\sigma \in \Gamma (\mathcal{E}[1])$ are called conformal scales \cite{governurowski}.  Continuing, there exists a canonical section $\boldsymbol{g}$ of $\mathcal{E}_{(ab)}[2]$ called the conformal metric which has the property that choosing a non-vanishing conformal scale $\sigma \in \Gamma (\mathcal{E}[1])$ defines a unique metric $g$ in the conformal class according to
\begin{equation}\label{eq:confmetric}
g := \sigma^{-2}\boldsymbol{g}.
\end{equation}
The conformal metric is the \emph{de facto} object which is used to identify $\mathcal{E}^{a}$ with $\mathcal{E}_{a}[2]$ and so raise and lower indices.  Notice that given a non-vanishing conformal scale $\sigma$, one may consider a semi-Riemannian manifold $(M,g)$ to be a triple $(M,c,\sigma)$, since the semi-Riemannian metric $g$ can be recovered from equation \eqref{eq:confmetric}.\\ \\
We now proceed to recall the rudiments of tractor calculus.  It is a fact that over any $n$-dimensional semi-Riemannian conformal manifold $(M,[g])$ of signature $(p,q)$ there exists a rank $(n+2)$ vector bundle $\mathcal{T} \to M$ called the standard tractor bundle which, upon choosing a metric $g$ in the conformal class, decomposes into a direct sum of conformal density bundles:
\begin{equation}\label{eq:confhol-tractorbundle}
\mathcal{T} = \mathcal{E}^{A} \overset{g}{:=} \mathcal{E}[1] \oplus \mathcal{E}^{a}[1] \oplus \mathcal{E}[-1].
\end{equation}
Sections of both the tractor bundle and its various tensor powers are referred to as tractor fields (or simply tractors), and are adorned with upper-case Latin indices.  One can see from the bundle structure of \eqref{eq:confhol-tractorbundle} that sections of the standard tractor bundle are triples:
\begin{equation}\label{eq:confchangetractor}
\Gamma(\mathcal{T}) \ni V^{A} :=
\left( \begin{array}{c}
\sigma \\ X^{a} \\ y
\end{array}{} \right),
\end{equation}
where $\sigma$ is referred to as the primary part, $X^{a}$ the secondary part and so on.  Now, if the chosen metric $g_{ab}$ is conformally transformed to $\wh{g}_{ab} = e^{2 \Upsilon} g_{ab}$, then the tractor $(\sigma, X^{a}, y)^{\perp}$ is related to its counterpart in the new scale defined by $\wh{g}$ according to
\begin{equation}
\left( \begin{array}{c}
\wh{\sigma} \\ \wh{X}^{a} \\ \wh{y}
\end{array}{} \right)
=
\left( \begin{array}{c}
\sigma \\ X^{a} + \Upsilon^{a} \sigma \\ y - \Upsilon_{b} X^{b} - \frac{1}{2} \Upsilon_{b} \Upsilon^{b} \sigma
\end{array}{} \right),
\end{equation}
where $\Upsilon_{a} := \nabla_{a} \Upsilon$.  The tractor bundle is equipped with a conformally invariant metric $h_{AB}$ of signature $(p+1,q+1)$ called the tractor metric, and we denote by $\la ~,~ \ra$ the associated tractor inner product.   In terms of two tractors $V^{A} = (\sigma, X^{a}, y)^{\perp} ~~\text{and } ~~ V^{B} = (\varsigma, Z^{b}, w)^{\perp}$ we have:
\begin{equation}\label{eq:tractormetric}
h_{AB} V^{A} V^{B} = V^{A} V_{A} = \sigma w + X^{a} Z_{a} + y \varsigma.
\end{equation}
The tractor bundle is endowed with a (tractor) metric compatible connection $\nabla^{\mathcal{T}}$ called the tractor connection.  With respect to the splitting \eqref{eq:confhol-tractorbundle}, the tractor connection is defined by:
\begin{equation}\label{eq:tractorconn}
\nabla_{b}^{\mathcal{T}} \left(
\begin{array}{c}
\sigma \\ X^{a} \\ y
\end{array}{}
\right)
:= \left(
\begin{array}{c}
\nabla_{b} \sigma - X_{b} \\ \nabla_{b} X^{a} + \delta_{b}^{\p{b}a}y + \mathsf{P}_{b}^{\p{b}a} \sigma \\ \nabla_{b} y - \mathsf{P}_{bc} X^{c}
\end{array}{}
\right)
\end{equation}
for $V^{A} = (\sigma, X^{a}, y)^{\perp}$, where $\nabla$ is the Levi-\v{C}ivit\`{a} connection of the chosen metric $g$ and $\mathsf{P}_{ab}$ is the trace-adjusted Ricci tensor, or Schouten tensor (of the conformal metric $\boldsymbol{g}$),
\begin{equation}\label{eq:schouten}
\mathsf{P}_{ab} = \frac{1}{n-2} \left( \mathsf{Ric}_{ab} - \mathsf{J} \boldsymbol{g}_{ab} \right),
\end{equation}
where $\mathsf{J} := \mathsf{P}^{c}_{\p{c}c}$ is the conformal metric trace of the Schouten tensor (rather than the $g$-trace).  Using equation \eqref{eq:confchangetractor}, it is straightforward to show that $\nabla^{\mathcal{T}}$ is a conformally invariant connection:
\begin{equation}
\wh{\nabla}_{b}^{\mathcal{T}} \left(
\begin{array}{c}
\wh{\sigma} \\ \wh{X}^{a} \\ \wh{y}
\end{array}{}
\right)
=
\left(
\begin{array}{c}
\wh{\nabla}_{b} \wh{\sigma} - \wh{X}_{b} \\ \wh{\nabla}_{b} \wh{X}^{a} + \delta_{b}^{\p{b}a}\wh{y} + \wh{\mathsf{P}}_{b}^{\p{b}a} \wh{\sigma} \\ \wh{\nabla}_{b} \wh{y} - \wh{\mathsf{P}}_{bc} \wh{X}^{c}
\end{array}{}
\right)
=
\left(
\begin{array}{c}
\nabla_{b} \sigma - X_{b} \\ \nabla_{b} X^{a} + \delta_{b}^{\p{b}a}y + \mathsf{P}_{b}^{\p{b}a} \sigma \\ \nabla_{b} y - \mathsf{P}_{bc} X^{c}
\end{array}{}
\right)
=
\nabla_{b}^{\mathcal{T}} \left(
\begin{array}{c}
\sigma \\ X^{a} \\ y
\end{array}{}
\right).
\end{equation}
We say that a tractor $V$ is a parallel tractor if and only if it is parallel with respect to the tractor connection:
\begin{equation}\label{eq:paralleltractor}
\nabla^{\mathcal{T}} V =0.
\end{equation}
To conclude our discussion of the modern approach to conformal geometry, we note that there is an important relation between parallel tractors and Einstein metrics in the conformal class:
\begin{theorem}[\cite{baumjuhl}]\label{thm-paralleleinstein}
Let $(M,c)$ be a conformal manifold of dimension $\geq 3$.  If the conformal class $c$ contains an Einstein metric $g$, then there exists a non-trivial section $V \in \Gamma(\mathcal{T})$ with $\nabla^{\mathcal{T}} V =0$.  On the other hand, if we suppose that there exists a nontrivial $\nabla^{\mathcal{T}}$-parallel section $V \in \Gamma(\mathcal{T})$, then there is an open dense subset $\wt{M} \subset M$ and an Einstein metric $g$ in the conformal class $c|_{\wt{M}}$.  In both cases, for the scalar curvature $\mathsf{scal}(g)$ of $g$ it holds that:
\begin{equation}
\begin{array}{l}
\mathsf{scal}(g) > 0 \iff \la V,V \ra < 0, \\
\mathsf{scal}(g) = 0 \iff \la V,V \ra = 0, \\
\mathsf{scal}(g) < 0 \iff \la V,V \ra > 0. 
\end{array} 
\end{equation}
\hfill $\square$
\end{theorem}

\subsection{Conformal holonomy theory}\label{section:theory}
Let $\nabla^{E}$ be a connection on a vector bundle $E \to M$ (with $M$ connected) and consider a smooth curve  $\gamma : [a, b] \subset \mathbb{R} \to M$ connecting two points $x$ and $y$ of $M$.  It is a fact that for any $e \in E_{x}$, there is a unique vector field $X_{e}$ along $\gamma$ with initial value $X_{e}(a) = e$, which is parallel to $\gamma$ along its length:
\begin{equation}
\frac{\nabla^{E} X_{e}}{\ud t}(t) =0, ~\text{ for all }~ t \in [a,b].
\end{equation}
The map
\begin{align}
{\mathscr{P}}^{\nabla^{E}}_{\gamma} : E_{x} &\longrightarrow E_{y} \nonumber \\
e &\longmapsto  X_{e}(b)
\end{align}
is a vector space isomorphism called parallel transport.  For a smooth loop $\gamma$ centred at $x \in M$ (that is to say, a curve for which $x = \gamma(a) = \gamma(b)$), parallel transport ${\mathscr{P}}^{\nabla^{E}}_{\gamma}: E_{x} \longrightarrow E_{x}$ is an automorphism of the fibre and so, because $E_{x}$ is a vector space, ${\mathscr{P}}^{\nabla^{E}}_{\gamma} \in \text{GL}(E_{x})$.  Denoting by ${\mathcal{L}}(M,x)$ the set of all smooth loops at $x \in M$, the holonomy group of $(E, \nabla^{E})$ with respect to the base point $x$ is the Lie group of all parallel displacements along loops centred at $x$:
\begin{equation}
\text{Hol}_{x} \left( E, \nabla^{E} \right) := \left\{ {\mathscr{P}}^{\nabla^{E}}_{\gamma} ~\Big|~ \gamma \in {\mathcal{L}}(M, x) \right\} \subset \text{GL}(E_{x}).
\end{equation}
The reduced holonomy group of $(E, \nabla^{E})$ (with respect to the base point $x$) are those holonomies around contractible loops:
\begin{equation}
\text{Hol}^{0}_{x} \left( E, \nabla^{E} \right) := \left\{ {\mathscr{P}}^{\nabla^{E}}_{\gamma} ~\Big|~ \gamma \in {\mathcal{L}}(M, x) \text{ null homotopic} \right\}.
\end{equation}
For a simply connected base space $M$, all loops are contractible and so $\text{Hol}^{0}_{x} \left( E, \nabla^{E} \right) = \text{Hol}_{x} \left( E, \nabla^{E} \right)$ in this case.  It is a fact that holonomy groups with respect to different base points are conjugate to eachother, which is to say that for a smooth curve $\gamma$ connecting $x$ and $y$,
\begin{equation}
\text{Hol}_{y} \left( E, \nabla^{E} \right) = \mathscr{P}^{\nabla^{E}}_{\gamma} \circ \text{Hol}_{x} \left( E, \nabla^{E} \right) \circ \mathscr{P}^{\nabla^{E}}_{\gamma^{-1}}.
\end{equation}
Therefore, the holonomy group $\text{Hol} \left( E, \nabla^{E} \right)$ is regarded as a conjugated class of groups without reference to a particular basepoint.  The representation $\rho: \text{Hol} \left( E, \nabla^{E} \right) \to \text{GL}(E_{x})$ is called the holonomy representation.  The action of a holonomy group $\text{Hol}(E, \nabla^{E})$ on its holonomy representation $E_{x}$ may leave a vector \emph{invariant}. A vector $w \in E_{x}$ is said to be holonomy invariant if
\begin{equation}\label{eq:holinvariantcondition}
\mathscr{P}^{\nabla^{E}}_{\gamma}(w) = w \text{ for all } \mathscr{P}^{\nabla^{E}}_{\gamma} \in \text{Hol}(E, \nabla^{E}).
\end{equation}
Analogously, a vector subspace $V$ of the holonomy representation is said to be holonomy invariant if condition \eqref{eq:holinvariantcondition} holds for all $v \in V$, and we denote this simply by
\begin{equation}
\text{Hol}(E, \nabla^{E})V = V.
\end{equation}
If the holonomy representation splits into the direct sum of (at least) two non-degenerate holonomy invariant subspaces, then it is said to be decomposable, and is indecomposable otherwise. It is worth noting that an irreducible representation is indecomposable, but an indecomposable representation is merely weakly irreducible.\\ \\
With this background in place, we proceed to discuss the related concept (and useful tool) of metric holonomy.  For a semi-Riemannian manifold $(M,g)$ of signature $(p,q)$, the metric holonomy group with respect to the basepoint $x$ is the holonomy group of the Levi-\v{C}ivit\`{a} connection $\nabla$ on the tangent bundle:
\begin{equation}
\text{Hol}_{x}(M,g) := \text{Hol}_{x}(TM, \nabla).
\end{equation}
Since the Levi-\v{C}ivit\`{a} connection preserves the metric, so does parallel transport.  Therefore the map ${\mathscr{P}}^{\nabla}_{\gamma}: T_{x}M \to T_{x}M$ is a linear isometry, and so
\begin{equation}
\text{Hol}_{x}(M,g) \subset O(T_{x}M, g_{x}).
\end{equation}
With respect to an orthogonal basis for $T_{x}M$, the metric holonomy group of $M$ may be regarded as a conjugated class of subgroups of the semi-orthogonal group $O(p,q)$ (again, without reference to a particular base point).\\ \\
We will now discuss the classification of metric holonomy groups.  To begin, notice that any weakly irreducible metric holonomy representation of a Riemannian manifold is in fact irreducible because $(T_{x}M, g_{x}) = (\mathbb{R}^{n}, g_{x})$ with $g_{x}$ positive definite does not admit any degenerate subspaces.  However, for a pseudo-Riemannian manifold of signature $(p,q)$, the holonomy representation $T_{x} M = \mathbb{R}^{p,q}$ may admit a degenerate, holonomy invariant subspace.  It is the presence of these degenerate, holonomy invariant subspaces that complicates the classification of pseudo-Riemannian metric holonomy (\emph{c.f.} \cite{kath, baumreview} and \cite{galaevleistner}.)  However, the metric holonomy groups of (simply connected) irreducible semi-Riemannian manifolds are completely classified and we shall recall some facts which will be useful throughout, focussing on the Riemannian theory for direct applicability to information geometry.  To begin, recall that the holonomy representation of a symmetric space is given by the isotropy representation:
\begin{theorem}[\cite{baumreview}]\label{thm-symmetricspaces}
Let $(M,g)$ be a symmetric space and let $G \subset \text{Isom}(M)$ be its transvection group.  Furthermore, let $\lambda : H \to \text{GL}(T_{x}M)$ be the isotropy representation of the stabilizer $H = G_{x}$ of a point $x \in M$.  Then
\begin{equation}
\lambda(H) = \text{Hol}_{x}(M,g).
\end{equation}
In particular, the holonomy group $\text{Hol}_{x}(M,g)$ is isomorphic to the stabilizer $H$ and, using this isomorphism, the holonomy representation $\rho$ is given by the isotropy representation $\lambda$. \p{.} \hfill $\square$
\end{theorem}
Since irreducible (locally) symmetric spaces are completely classified\footnote{The list, which describes pairs $(G,H)$ together with the isotropy representation $\lambda$, may be found in the pseudonymous Besse's monograph \cite{besse}.} (\emph{i.e.} those symmetric spaces with irreducible isotropy representation) so are their holonomy groups by theorem \ref{thm-symmetricspaces}.  Consequently, to complete the classification of irreducible metric holonomy, only the non-locally symmetric case remains.  This was completed by Berger in \cite{berger}:
\begin{theorem}[Berger \cite{baumreview}]
Let $(M^{n},g)$ be an $n$-dimensional, simply connected, irreducible, non-locally symmetric Riemannian manifold.  Then the holonomy group $\text{Hol}(M^{n},g)$ is up to conjugation in $O(n)$ either $SO(n)$ or one of the following groups with its standard representation:
\begin{table}[!htbp]
\centering
\caption{Riemannian Berger list} \label{table:bergerriemannian}
\vspace{1mm}
\begin{tabular}{lll}
  \toprule[1.5pt]
Dimension & Holonomy group & Special geometry \\
  \midrule
$2m \geq 4$ & $U(m)$ & K\"{a}hler manifold \vspace{1mm} \\
$2m \geq 4$ & $SU(m)$ & Ricci-flat K\"{a}hler manifold  \vspace{1mm} \\
$4m \geq 8$ & $\text{Sp}(m)$ & Hyperk\"{a}hler manifold   \vspace{1mm} \\
$4m \geq 8$ & $\text{Sp}(m) \cdot \text{Sp}(1)$ & Quaternionic K\"{a}hler manifold  \vspace{1mm} \\
7 & $G_{2}$ & $G_{2}$-manifold  \vspace{1mm} \\
8 & $\text{Spin}(7)$ & $\text{Spin}(7)$-manifold \vspace{1mm} \\
 \bottomrule[1.5pt]
\end{tabular} 
\end{table}
\end{theorem}
With this motivation given, we will now describe the construction of our primary object of study, the conformal holonomy group.  The construction follows in a similar manner to the metric holonomy case above, and our presentation follows \cite{leitner}.  Consider a semi-Riemannian conformal manifold $(M,[g])$.  For some tractor $T_{o} \in \mathcal{T}_{x}$ and a loop $\gamma \in \mathcal{L}(M,x)$, we see from our previous discussion that there exists a parallel tractor field $T$ along $\gamma$, in other words: $(\nabla^{\mathcal{T}} T_{o}/\ud t)(t) =0$ for all $t \in [a,b]$. The parallel transport ${\mathscr{P}}^{\nabla^{\mathcal{T}}}_{\gamma}: \mathcal{T}_{x} \to \mathcal{T}_{x}$ is an automorphism, and since the tractor connection $\nabla^{\mathcal{T}}$ preserves the tractor metric $\la \cdot, \cdot \ra$, the set of these automorphisms ${\mathscr{P}}^{\nabla^{\mathcal{T}}}_{\gamma}$ is a subgroup of $O(\mathcal{T}_{x}, \la \cdot, \cdot \ra_{x})$.  Consequently, the holonomy group of the tractor connection is defined as
\begin{equation}
 \text{Hol}_{x}(\mathcal{T}, \nabla^{\mathcal{T}}) := \left\{ {\mathscr{P}}^{\nabla^{\mathcal{T}}}_{\gamma} ~\Big|~ \gamma \in {\mathcal{L}}(M, x) \right\} \subset O \left( \mathcal{T}_{x}, \la \cdot, \cdot \ra_{x} \right).
\end{equation}
For brevity however, and following the standard terminology, the conformal holonomy group (with respect to the base point $x$) is defined to be the holonomy group of the tractor connection:
\begin{equation}
\text{Hol}_{x}(M,[g]) :=  \text{Hol}_{x}(\mathcal{T}, \nabla^{\mathcal{T}})  \subset O \left( \mathcal{T}_{x}, \la \cdot, \cdot \ra_{x} \right).
\end{equation}
As previous, we omit the basepoint and consider the conformal holonomy group $\text{Hol}(M,[g])$ to be defined upto conjugation.  Choosing an orthogonal basis for the fibre $\mathcal{T}_{x}$, the conformal holonomy group may be regarded as a subgroup of the semi-orthogonal group:
\begin{equation}\label{eq:defconfhol}
\text{Hol}(M,[g]) \subset O(p+1,q+1),
\end{equation}
analogous to the metric holonomy group of a semi-Riemannian manifold.  Our results are obtained by studying the action of the conformal holonomy group on its holonomy representation, the conformal holonomy representation\footnote{For brevity, we will sometimes refer to the conformal holonomy representation simply as the \emph{holonomy representation} when the context is appropriate.}, and so the following facts will be useful:
\begin{corollary}[\cite{baumreview}]\label{baumholinv}
Let $\text{Hol}(E, \nabla^{E})$ be a holonomy group and $V$ a proper subspace of the holonomy representation.  If $V$ is $\text{Hol}(E, \nabla^{E})$-invariant then the orthogonal complement $V^{\perp}$ is $\text{Hol}(E, \nabla^{E})$-invariant as well.  Furthermore, if $V$ is non-degenerate then $V^{\perp}$ is non-degenerate as well and the holonomy representation is decomposable.
\end{corollary}
Notice that for a semi-Riemannian conformal manifold of signature $(p,q)$, the holonomy representation $\mathcal{T}_{x} = \mathbb{R}^{p+1,q+1}$ and so degenerate, holonomy invariant subspaces $Z \subset \mathcal{T}_{x}$ can (and do) exist. The existence of a degenerate, holonomy invariant subspace of the conformal holonomy representation imposes stringent conditions on the geometry of the conformal manifold, and so these subspaces are by no means generic.   The following fact is useful for dealing with such degenerate subspaces:
\begin{corollary}[\cite{galaevleistner}]\label{thm-degeneratesubspaces}
Let $\text{Hol}(E, \nabla^{E})$ be a holonomy group and $Z$ a degenerate, proper $\text{Hol}(E, \nabla^{E})$-invariant subspace of the holonomy representation.  Then $Z \cap Z^{\perp}$ is a totally isotropic, $\text{Hol}(E, \nabla^{E})$-invariant subspace.
\end{corollary}
For the conformal holonomy representation $\mathcal{T}_{x} = \mathbb{R}^{p+1,q+1}$, notice that the dimension of a putative totally isotropic subspace $Z \cap Z^{\perp}$ ranges from 1 to $\text{min}\{ p + 1, q + 1 \}$.  For example,  the holonomy representation $\mathbb{R}^{1,n+1}$ of a Riemannian conformal manifold $(M^{n},[g])$ admits at most null $\text{Hol}(M^{n},[g])$-invariant lines.  On the other hand, the holonomy representation $\mathbb{R}^{2,n}$ of a Lorentzian conformal manifold $(M^{1,n-1},[g])$ can admit totally isotropic $\text{Hol}(M^{1,n-1},[g])$-invariant planes as well as null $\text{Hol}(M^{1,n-1},[g])$-invariant lines. Since the classification of conformal holonomy relies on results from the classification of semi-Riemannian metric holonomy, it too is complicated by the existence of totally isotropic holonomy invariant subspaces. \\ \\
We say that a conformal manifold $(M^{p,q},[g])$ of dimension $\geq 3$ is conformally indecomposable if the tractor connection preserves at most a single line bundle and its orthogonal complement, but nothing else \cite{armstrong}.  When a conformal manifold is conformally decomposable (\emph{i.e.} the tractor connection preserves a $k \geq 2$-dimensional non-degenerate holonomy invariant subspace) it is decomposed into a product of Einstein spaces.  This may be viewed as a conformal holonomy analogue of the de Rham splitting theorem:
\begin{theorem}[\cite{baumjuhl}]\label{einsteinproductifinvsubspace}
Let $(M^{p,q},c)$ be a simply connected conformal manifold of dimension $n \geq 3$ and suppose that there exists a $k$-dimensional non-degenerate $\text{Hol}(M^{p,q},c)$-invariant subspace $V^{k} \subset \mathcal{T}_{x}$, where $2 \leq k \leq n$.  Then, any point of a certain open dense subset $\wt{M}^{p,q} \subset M^{p,q}$ has a neighbourhood $U(y)$ with a metric $g \in c|_{U(y)}$, such that $(U(y),g)$ is isometric to a product $(N_{1}, g_{1}) \times (N_{2}, g_{2})$, where $(N_{i}, g_{i})$ are Einstein spaces of dimension $k-1$ and $n-(k-1)$, respectively.  If $k \neq 2, n$ then the scalar curvatures satisfy
\begin{equation}
\mathsf{scal}(g_{1}) = -\frac{(k-1)(k-2)}{(n-k+1)(n-k)}\mathsf{scal}(g_{2}) \neq 0.
\end{equation}
\hfill $\square$
\end{theorem}
We now recall a remarkable (and very useful) fact in conformal holonomy theory:  Einstein metrics in the conformal class of a conformal manifold are in bijective correspondence with holonomy invariant lines in the holonomy representation.  The precise relationship is given by the well-known theorem below:
\begin{theorem}[\cite{alt}]\label{alt}
Let $(M,c)$ be an $n$-dimensional conformal manifold of signature $(p,q)$ and $\wt{M}$ the open dense subset on which the conformal scale is non-vanishing.  Then  $\text{Hol}(M,[g])$-invariant lines $\mathbb{R}V \subset \mathbb{R}^{p+1,q+1}$ are in bijective correspondence with Einstein metrics $g \in c|_{\wt{M}}$ defined up to singularity on $M$.  Under this correspondence, we have:
\begin{equation}
\begin{array}{l}
\mathsf{scal}(g) > 0 \iff \la V,V \ra < 0, \\
\mathsf{scal}(g) = 0 \iff \la V,V \ra = 0, \\
\mathsf{scal}(g) < 0 \iff \la V,V \ra > 0.
\end{array}
\end{equation}
\hfill $\square$
\end{theorem}
With reference to theorem \ref{thm-paralleleinstein}, we see that holonomy invariant lines are in bijective correspondence with parallel tractors.  This is a manifestation of the holonomy principle \cite{baumreview}, which is a bijective correspondence between holonomy invariant tensors and parallel sections of the tensor bundle in question.  The conformal holonomy group of a conformally Einstein space $(M^{p,q},[g])$ in fact coincides with the metric holonomy group of a metric cone over the Einstein space $(M^{p,q},g)$: 
\begin{theorem}[\cite{baumjuhl}]\label{feffermangraham}
Let $(M^{p,q},g)$ be an Einstein space of signature $(p,q)$ and dimension $n = p + q \geq 3$ with scalar curvature $\mathsf{scal}(g) \neq 0$, and let $(C(M^{p,q}),g_{C})$ denote the cone $C(M^{p,q}) := M^{p,q} \times \mathbb{R}^{+}$ over $(M^{p,q},g)$ with Ricci-flat metric
\begin{equation}\label{eq:conemetric}
(g_{C})_{(x,t)} := \text{sign}(\mathsf{scal}(g)) \left( \ud t^{2} + \frac{\mathsf{scal}(g)}{n(n-1)} t^{2} g_{x} \right).
\end{equation}
Then the holonomy groups satisfy
\begin{equation}\label{eq:confholmathcesconehol}
\text{Hol}_{x}(M^{p,q},[g]) = \text{Hol}_{x,1}(C(M^{p,q}),g_{C}).
\end{equation}
\hfill $\square$
\end{theorem}
Moreover, the conformal holonomy groups of simply connected, conformally indecomposable Riemannian Einstein spaces were classified by Armstrong \cite{armstrong}:
\begin{theorem}[\cite{baumjuhl}]\label{armstrong}
Let $(M,g)$ be a simply connected, conformally indecomposable Riemannian Einstein space of dimension $n \geq 4$, and let $\text{Hol}(M,[g]) \subset SO^{0}(1,n+1)$ be its conformal holonomy group.
\begin{enumerate}
\item{If $\mathsf{scal}(g) < 0$, then $\text{Hol}(M,[g]) = SO^{0}(1,n)$,}
\item{If $\mathsf{scal}(g) > 0$, then $\text{Hol}(M,[g])$ is one of the following groups:\\
$SO(n+1)$, $SU(\frac{n+1}{2})$, $\text{Sp}(\frac{n+1}{4})$, $G_{2}$ if $n=6$, or $\text{Spin}(7)$ if $n=7$,}
\item{If $\mathsf{scal}(g) = 0$, then $\text{Hol}(M,[g])$ is one of the following groups:\\
$SO(n) \ltimes \mathbb{R}^{n}$, $SU(\frac{n}{2}) \ltimes \mathbb{R}^{n}$, $\text{Sp}(\frac{n}{4}) \ltimes \mathbb{R}^{n}$, $G_{2} \ltimes \mathbb{R}^{n}$ if $n=7$, or $\text{Spin}(7) \ltimes \mathbb{R}^{n}$ if $n=8$.}
\end{enumerate}
\hfill $\square$
\end{theorem}
In theorem \ref{armstrong}, $SO^{0}(1,n)$, $SO(n+1)$ and $SO(n) \ltimes \mathbb{R}^{n}$ are the generic conformal holonomy groups, while the other groups in cases \textit{2} and \textit{3} correspond to additional geometric structures on the Einstein space.  A list of these structures may be found in \cite{alt}.

\section{Conformal holonomy of the bivariate Gaussian manifold}\label{section:confholg}
\subsection{Conformal holonomy of the bivariate Gaussian manifold}\label{subsection:confholg}
Let $\left( \mathcal{G}, g \right)$ denote the bivariate Gaussian $5$-manifold with Fischer-Rao metric, $g$, and let $\left( \mathcal{G}, [g] \right)$ denote the associated conformal manifold.  It is a fact that the conformal holonomy group of a conformal manifold is trivial if the manifold is conformally flat.  A necessary and sufficient condition for conformal flatness is that the Weyl tensor vanishes:
\begin{align}\label{eq:weyltensor}
\mathsf{C}_{abcd} =~ &\mathsf{R}_{abcd} - \frac{1}{n-2} \big( \mathsf{Ric}_{ad} g_{bc} -  \mathsf{Ric}_{ac} g_{bd} +  \mathsf{Ric}_{bc} g_{ad} - \mathsf{Ric}_{bd} g_{ac} \big) \nonumber \\ 
&+ \frac{1}{(n-1)(n-2)} \mathsf{scal} \big( g_{ac} g_{bd} - g_{ad} g_{bd} \big) = 0.
\end{align}
By direct calculation, the following facts are immediate:
\begin{proposition}\label{prop:confflat}
$\left( \mathcal{G}, g \right)$ is non-conformally flat.
\end{proposition}
\begin{proof}
It is sufficient to show that a single component of the Weyl tensor is non-zero to distinguish it from the $0$-tensor. Without loss of generality, consider $\mathsf{C}_{1234}$. Substituting the expressions for the Fischer-Rao metric \eqref{eq:fisherrao}, the Riemann tensor \eqref{eq:fr-riemann} and the Ricci tensor \eqref{eq:fr-ricci} into \eqref{eq:weyltensor} and simplifying the resulting expression yields $\mathsf{C}_{1234} = - \sigma_{2} / 4 \Delta^{2}$, and so $\mathcal{G}$ is non-conformally flat as claimed.
\end{proof}
\begin{proposition}\label{prop:notconfeinstein}
$\left( \mathcal{G}, g \right)$ is not a conformally Einstein space.
\end{proposition}
\begin{proof}
Consider the conformal transformation of the Fischer-Rao metric, $g_{ab}$:
\begin{equation}\label{eq:conftransg}
g_{ab} \longmapsto \wh{g}_{ab} = e^{2\Upsilon}g_{ab}.
\end{equation}
We will show that there does not exist $\Upsilon \in C^{\infty}(\mathcal{G})$ such that $\wh{g}_{ab}$ is an Einstein metric. Suppose towards a contradiction that $\wh{g} \in [g]$ is an Einstein metric, which is to say that
\begin{equation}\label{eq:einsteinconditionstar}
\wh{\mathsf{Ric}}_{ab} = \frac{1}{5} \p{.} \wh{\mathsf{scal}} \p{.} \wh{g}_{ab}
\end{equation}
in five dimensions. Substituting the well known formulae for the conformal transformation of the Ricci tensor and scalar curvature into \eqref{eq:einsteinconditionstar} and simplifying yields
\begin{equation}
\mathsf{Ric}_{ab} = 3 \left( \nabla_{a} \Upsilon_{b} - \Upsilon_{a} \Upsilon_{b}  \right) + \left[ \square \Upsilon + 3 \Upsilon^{a} \Upsilon_{a} - \frac{e^{-4 \Upsilon}}{5} \left( \mathsf{scal} - \frac{16}{3} e^{-\frac{3}{2}\Upsilon} \square \left( e^{\frac{3}{2}\Upsilon} \right) \right) \right] g_{ab}.
\end{equation} 
Without loss of generality, evaluating the $11$-component and $33$-components of the Ricci tensor and Fischer-Rao metric yields $\sigma_{2}/\Delta - \sigma_{2}^{\p{2}2}/\Delta^{2} =0$, which clearly has no solutions $\Upsilon$ for $\Delta \neq 0$, and so $g$ is not conformally Einstein as claimed. 
\end{proof}
The following facts are also easily established:
\begin{proposition}\label{prop:confindecomp}
$\left( \mathcal{G}, g \right)$ is conformally indecomposable.
\end{proposition}
\begin{proof}
Recall from theorem \ref{einsteinproductifinvsubspace} that the conformal class of a conformally decomposable manifold is locally represented by a product of Einstein spaces.  Suppose towards a contradiction that there exists a conformal transformation of the Fischer-Rao metric $g \longmapsto \wh{g} = e^{2\Upsilon}g$ such that $\wh{g}$ is a product metric.  The Riemann tensor of a product manifold is necessarily block diagonal, but the non-zero components \eqref{eq:fr-riemann} of $\mathsf{R}$ are not block diagonal, so we arrive at a contradiction and $\left( \mathcal{G}, g \right)$ is conformally indecomposable as claimed.
\end{proof}
\begin{proposition}[\cite{li}]\label{nsimplyconnected}
$\mathcal{G}$ is simply connected.
\end{proposition}
\begin{proof}
By \eqref{eq:defg}, $\mathcal{G}$ is homeomorphic to $\mathbb{R}^{2} \times \text{PD}(d,\mathbb{R}^{2})$ and $\mathbb{R}^{2}$ is clearly contractible,  so $\pi_{1}(\mathbb{R}^{2}) = 0$.  Now, for $A, B \in \text{PD}(d,\mathbb{R}^{2})$, $(1-t)A + tB \in \text{PD}(d,\mathbb{R}^{2})~ \forall t \in [0,1]$ and so $\text{PD}(d,\mathbb{R}^{2})$ is convex and therefore contractible. Consequently, $\pi_{1}(\text{PD}(d,\mathbb{R}^{2})) = 0$ and so
\begin{equation}
\pi_{1}(\mathcal{G}) \cong \pi_{1} \left( \mathbb{R}^{2} \times \text{PD}(d,\mathbb{R}^{2}) \right) \cong \pi_{1} \left( \mathbb{R}^{2} \right) \times \pi_{1} \left( \text{PD}(d,\mathbb{R}^{2}) \right)  = 0
\end{equation}
as claimed.
\end{proof}
Lastly, let us recall the classification of Lie groups acting irreducibly on $\mathbb{R}^{1,n+1}$:
\begin{theorem}[\cite{discalaolmos}]\label{discalaolmos}
Let $H \subset SO(1,n+1)$ be a connected Lie group which acts irreducibly on $\mathbb{R}^{1,n+1}$.  Then $H = SO^{0}(1,n+1)$. \hfill $\square$
\end{theorem}
With all of the above understood, we can determine the conformal holonomy group of the bivariate Gaussian manifold:
\begin{theorem}\label{thm:mainresult1}
$\text{Hol}(\mathcal{G}, [g]) = SO^{0}(1,6)$.
\end{theorem}
\begin{proof}
Our strategy will be to show that $\text{Hol}(\mathcal{G}, [g]) \subset O(1,6)$ acts irreducibly on $\mathcal{T}_{x} = \mathbb{R}^{1,6}$ so that we may apply classification theorem \ref{discalaolmos} to determine the conformal holonomy group. Therefore, we must show that $\mathcal{T}_{x}$ admits no subrepresentations, or in other words: show that $\mathcal{T}_{x}$ admits no $\text{Hol}(\mathcal{G}, [g])$-invariant subspaces (degenerate or non-degenerate).  We begin with the degenerate case. Suppose towards a contradiction that $\mathcal{T}_{x}$ admits a degenerate, proper $\text{Hol}(\mathcal{G}, [g])$-invariant subspace, $Z$.  Then by corollary \ref{thm-degeneratesubspaces}, $Z \cap Z^{\perp}$ is a null $\text{Hol}(\mathcal{G}, [g])$-invariant line and so by theorem \ref{alt} there exists a scalar-flat Einstein metric on an open, dense subset of $(\mathcal{G}, [g])$.
However, $(\mathcal{G}, [g])$ is not conformally Einstein by proposition \ref{prop:notconfeinstein}, so we arrive at a contradiction and conclude that $\mathcal{T}_{x}$ admits no degenerate, proper $\text{Hol}(\mathcal{G}, [g])$-invariant subspaces.  Similarly, suppose towards a further contradiction that $\mathcal{T}_{x}$ admits a non-degenerate, proper $\text{Hol}(\mathcal{G}, [g])$-invariant line, $L^{1}$.  Then by the same theorem, there exists a non-scalar flat Einstein metric on an open, dense subset of $(\mathcal{G}, [g])$, which we have seen do not exist.  Notice that the orthogonal complement of a putative $6$-dimensional $\text{Hol}(\mathcal{G}, [g])$-invariant subspace is a holonomy invariant line by corollary \ref{baumholinv}, and so we likewise conclude that $\mathcal{T}_{x}$ admits no $6$-dimensional non-degenerate, proper $\text{Hol}(\mathcal{G}, [g])$-invariant subspaces either.  Lastly, and towards a final contradiction, suppose that $\mathcal{T}_{x}$ admits a non-degenerate, proper $\text{Hol}(\mathcal{G}, [g])$-invariant subspace $V^{k} \subset \mathcal{T}_{x}$, where $2 \leq k \leq 5$.  Then, by theorem \ref{einsteinproductifinvsubspace} any point of an open dense subset $\wt{\mathcal{G}} \subset \mathcal{G}$ has a neighbourhood $U(y)$ with a metric $\wt{g} \in [g]|_{U(y)}$ such that $(U(y),\wt{g})$ is isometric to a product $(N_{1}, g_{1}) \times (N_{2}, g_{2})$, where $(N_{i}, g_{i})$ are Einstein spaces of dimension $k-1$ and $5-(k-1)$, respectively. But $(\mathcal{G}, [g])$ is conformally indecomposable by proposition \ref{prop:confindecomp}, and so admits no such product metric and therefore $\mathcal{T}_{x}$ admits no $\text{Hol}(\mathcal{G}, [g])$-invariant subspaces $V^{k}$ for $2 \leq k \leq 5$.  Having shown that $\mathcal{T}_{x}$ admits no subrepresentations, we conclude that $\text{Hol}(\mathcal{G}, [g])$ acts irreducibly on $\mathcal{T}_{x} = \mathbb{R}^{1,6}$.  Now, since $(\mathcal{G}, [g])$ is simply connected by proposition \ref{nsimplyconnected}, $\text{Hol}(\mathcal{G}, [g]) = \text{Hol}^{0}(\mathcal{G}, [g])$ is clearly connected.  Moreover, since simply connected manifolds are necessarily orientable then $\text{Hol}(\mathcal{G}, [g]) \subset SO(1,6) \subset O(1,6)$. Therefore, theorem \ref{discalaolmos} applies and $\text{Hol}(\mathcal{G}, [g]) = SO^{0}(1,6)$ as claimed.
\end{proof}

\subsection{Conformal holonomy of the independence submanifold}
Following \cite{arwini}, the independence submanifold is defined by:
\begin{equation}
\mathcal{I} \subset \mathcal{G}: \sigma_{12} = 0.
\end{equation}
Its density functions are of the form:
\begin{equation}
f(x, y; \mu_{1}, \mu_{2}, \sigma_{1}, \sigma_{2}) = N_{X}(\mu_{1}, \sigma_{1}) \cdot N_{Y}(\mu_{2}, \sigma_{2}),
\end{equation}
which is to say that the random variables $X$ and $Y$ are statistically independent.  Consequently, $\mathcal{I}$ is a direct product whose factors are univariate Gaussian manifolds:
\begin{equation}\label{eq:productofunivariates}
\left\{ N_{X}(\mu_{1}, \sigma_{1}), \mu_{1} \in \mathbb{R}, \sigma_{1} \in \mathbb{R}^{+} \right\} ~\text{and}~ \left\{ N_{Y}(\mu_{2}, \sigma_{2}), \mu_{2} \in \mathbb{R}, \sigma_{2} \in \mathbb{R}^{+} \right\}.
\end{equation}
As a real $4$-manifold, $\mathcal{I}$ is endowed with a natural coordinate system
\begin{equation}
(\theta_{1},\theta_{2},\theta_{3},\theta_{4}) = \left( \frac{\mu_{1}}{\sigma_{1}},  \frac{\mu_{2}}{\sigma_{2}}, -\frac{1}{2\sigma_{1}}, -\frac{1}{2\sigma_{2}} \right),
\end{equation}
with corresponding potential function
\begin{equation}
\varphi(\theta) = \log \left( 2 \pi \sqrt{\Delta} \right) + \Delta \left( \theta_{2}^{\p{2}2}\theta_{3} + \theta_{1}^{\p{1}2} \theta_{4} \right),
\end{equation}
where $\Delta = \frac{1}{4 \theta_{3} \theta_{4}}$.  The Fischer-Rao metric on $\mathcal{I}$ is
\begin{equation}\label{eq:fisherraoindep}
g_{ab} = \left[
\begin{array}{cccc}
\sigma_{1} & 0 & 2 \mu_{1}\sigma_{1} & 0 \\
0 & \sigma_{2} & 0 & 2 \mu_{2}\sigma_{2} \\
2 \mu_{1}\sigma_{1} & 0 & 2\sigma_{1}\left(2\mu_{1}^{\p{1}2} + \sigma_{1} \right) & 0 \\
0 & 2 \mu_{2}\sigma_{2} & 0 & 2\sigma_{2}\left(2\mu_{2}^{\p{2}2} + \sigma_{2} \right)
\end{array}{}
\right],
\end{equation}
and the non-zero independent components of the Levi-\v{C}ivit\`{a} connection are:
\begin{align}
\Gamma_{13,1} &= \sigma_{1}^{\p{1}2},~
\Gamma_{33,1} = 4 \mu_{1} \sigma_{1}^{\p{1}2}, ~ 
\Gamma_{24,2} = \sigma_{2}^{\p{2}2},~
\Gamma_{44,2} = 4 \mu_{2} \sigma_{2}^{\p{2}2},~
\Gamma_{33,3} = 4\sigma_{1}^{\p{1}2} \left( 3 \mu_{2}^{\p{2}2} + \sigma_{1} \right), \nonumber \\
\Gamma_{11}^{\p{11}1} &= -\Gamma_{13}^{\p{13}3} = - \mu_{1},~
\Gamma_{31}^{\p{31}1} = \sigma_{1} - 2\mu_{1}^{\p{1}2},~
\Gamma_{11}^{\p{11}3} = \Gamma_{22}^{\p{221}4} = \frac{1}{2},~
\Gamma_{33}^{\p{33}1} = -4 \mu_{1}^{\p{1}3},~
\Gamma_{33}^{\p{33}3} = 2 \left( \sigma_{1} +\mu_{1}^{\p{1}2} \right),  \nonumber \\
\Gamma_{22}^{\p{22}2} &= -\Gamma_{24}^{\p{24}4} = - \mu_{2},~
\Gamma_{42}^{\p{42}2} = \sigma_{2} + 2\mu_{2}^{\p{2}2},~
\Gamma_{44}^{\p{44}2} = -4 \mu_{2}^{\p{2}3},~
\Gamma_{44}^{\p{44}4} = 2 \left( \sigma_{2} +\mu_{2}^{\p{2}2} \right). 
\end{align}
In contrast to the bivariate Gaussian manifold, the Riemann tensor $\mathsf{R}_{abcd}$ of the independence submanifold has only two non-vanishing independent components:
\begin{equation}\label{eq:riemanntensorindep}
\mathsf{R}_{1313} = - \sigma_{1}^{\p{1}3} ~\text{and}~ \mathsf{R}_{2424} = - \sigma_{2}^{\p{2}3}.
\end{equation}
The Ricci tensor is given by
\begin{equation}\label{eq:riccitensorindep}
\mathsf{Ric}_{ab} = - \left[
\begin{array}{cccc}
\frac{\sigma_{1}}{2} & 0 & \mu_{1}\sigma_{1} & 0  \\
0 & \frac{\sigma_{2}}{2} & 0 & \mu_{2}\sigma_{2}  \\
\mu_{1}\sigma_{1} & 0 & \sigma_{1} \left( 2 \mu_{1}^{\p{1}2} + \sigma_{1} \right)   & 0  \\
0 & \mu_{2}\sigma_{2} & 0 & \sigma_{2} \left( 2 \mu_{2}^{\p{2}2} + \sigma_{2} \right)
\end{array}{}
\right],
\end{equation}
while the scalar curvature $\mathsf{scal} = -2$.  Lastly, the sectional curvature is given by:
\begin{equation}\label{eq:sectionalcurvindep}
\mathsf{k} = - \frac{1}{2} \left[
\begin{array}{cccc}
0 & 0 & 1 & 0 \\
0 & 0 & 0 & 1 \\
1 & 0 & 0 & 0 \\
0 & 1 & 0 & 0 
\end{array}{}
\right].
\end{equation}
Comparing the metric tensor \eqref{eq:fisherraoindep} with the Ricci tensor \eqref{eq:riccitensorindep} and the scalar curvature, it is evident that the Fischer-Rao metric $g$ is an Einstein metric. We have previously seen (\emph{c.f.} theorem \ref{alt}) that the existence of an Einstein metric in the conformal class implies that the conformal holonomy representation admits a holonomy invariant subspace, that the representation is therefore reducible and the rank of the conformal holonomy group is reduced according to the sign of the scalar curvature of the Einstein metric $g \in c$.  So in contrast to the bivariate Gaussian manifold $(\mathcal{G}, g)$, the independence submanifold $(\mathcal{I}, g)$ has reducible conformal holonomy.\\ \\
As in section \ref{subsection:confholg}, we must establish some facts regarding the geometry and topology of $(\mathcal{I}, g)$ in order to determine its conformal holonomy group.
\begin{proposition}\label{indepsimplyconnected}
$\mathcal{I}$ is simply connected.
\end{proposition}
\begin{proof}
Topologically, $\mathcal{I}$ is a direct product whose factors are univariate Gaussian manifolds \eqref{eq:productofunivariates}, which are homeomorphic to $\mathbb{R} \times \text{PD}(d,\mathbb{R})$. $\mathbb{R}$ is clearly contractible,  so $\pi_{1}(\mathbb{R}) = 0$.  Now, for $A, B \in \text{PD}(d,\mathbb{R})$, $(1-t)A + tB \in \text{PD}(d,\mathbb{R})~ \forall t \in [0,1]$ and so $\text{PD}(d,\mathbb{R})$ is convex and therefore contractible. Consequently, $\pi_{1}(\text{PD}(d,\mathbb{R})) = 0$ and so
\begin{align}
\pi_{1}(\mathcal{I}) &\cong \pi_{1} \big( \mathbb{R} \times \text{PD}(d,\mathbb{R}) \times \mathbb{R} \times \text{PD}(d,\mathbb{R}) \big) \nonumber \\
&\cong \pi_{1} \left( \mathbb{R} \right) \times \pi_{1} \left( \text{PD}(d,\mathbb{R}) \right) \times \pi_{1} \left( \mathbb{R} \right) \times \pi_{1} \left( \text{PD}(d,\mathbb{R}) \right) =0
\end{align}
as claimed.
\end{proof}
\begin{proposition}\label{prop:notconfflatindep}
$\mathcal{I}$ is non-conformally flat.
\end{proposition}
\begin{proof}
The proof is conceptually identical to that of proposition \ref{prop:confflat}. Without loss of generality, consider the $1234$-component of the $(0,4)$-valent Weyl tensor, $\mathsf{C}_{abcd}$. Substituting the expressions for the Fischer-Rao metric \eqref{eq:fisherraoindep}, Ricci tensor \eqref{eq:riemanntensorindep} and scalar curvature into \eqref{eq:weyltensor} and simplifying the resulting expression yields
\begin{equation}
\mathsf{C}_{1234} = -\frac{7}{3} \sigma_{1}\sigma_{2}\left( 2 \mu_{1}^{\p{1}2} + \sigma_{1}\right)\left( 2 \mu_{2}^{\p{2}2} + \sigma_{2}\right) \neq 0,
\end{equation}
and so $\mathcal{I}$ is non-conformally flat as claimed.
\end{proof}
\begin{proposition}\label{prop:onlyconfeinsteinmetric}
The Fischer-Rao metric $g$ is the only Einstein metric in its conformal class.
\end{proposition}
\begin{proof}
It follows from Brinkmann's celebrated results on conformal maps between Einstein spaces that if a Riemannian Einstein $4$-manifold $(M, g)$ admits a conformal transformation $g \longmapsto \wh{g} = e^{2\Upsilon}g $ where $\wh{g}$ is also Einstein, then $(M,g)$ has constant sectional curvature \cite{brinkmann} (see also \cite{kuhnelrademacher}).  By equation \eqref{eq:sectionalcurvindep}, the sectional curvature $\mathsf{k}$ is non-constant, and so $g$ is the only Einstein metric in the conformal class as claimed.
\end{proof}
\begin{theorem}
$\text{Hol}(\mathcal{I}, [g]) = SO^{0}(1,4)$.
\end{theorem}
\begin{proof}
The proof is largely analogous to that of theorem \ref{thm:mainresult1}, but in this instance, the existence of an Einstein metric in the conformal class implies that $\mathcal{T}_{x}$ admits a subrepresentation: a non-degenerate, proper $\text{Hol}(\mathcal{I}, [g])$-invariant line $L^{1}$ with positive square norm. Suppose towards a contradiction that in addition to $L^{1}$, $\mathcal{T}_{x}$ admits a degenerate, proper $\text{Hol}(\mathcal{I}, [g])$-invariant subspace, $Z$.  Then by corollary \ref{thm-degeneratesubspaces}, $Z \cap Z^{\perp}$ is a null $\text{Hol}(\mathcal{I}, [g])$-invariant line and so by theorem \ref{alt} there exists a scalar-flat Einstein metric on an open, dense subset of $(\mathcal{I}, [g])$.  However, $g \in [g]$ is the only Einstein metric in the conformal class by proposition \ref{prop:onlyconfeinsteinmetric}, so we arrive at a contradiction and conclude that $\mathcal{T}_{x}$ admits no degenerate, proper $\text{Hol}(\mathcal{I}, [g])$-invariant subspaces. By the same reasoning, $\mathcal{T}_{x}$ cannot admit any non-degenerate, proper $\text{Hol}(\mathcal{I}, [g])$-invariant lines (and their $5$-dimensional holonomy invariant orthogonal complements) besides $L^{1}$.  Suppose towards a further contradiction that the conformal holonomy representation preserves a $k$-dimensional non-degenerate $\text{Hol}_{x}(\mathcal{I}, [g])$-invariant subspace $V^{k} \subset \mathcal{T}_{x}$, where $2 \leq k \leq 4$.  Then, by theorem \ref{einsteinproductifinvsubspace} any point of an open dense subset $\mathcal{D} \subset \mathcal{I}$ has a neighbourhood $U(y)$ with a metric $\wt{g} \in [g]|_{U(y)}$ such that $(U(y),\wt{g})$ is isometric to a product $(N_{1}, g_{1}) \times (N_{2}, g_{2})$, where $(N_{i}, g_{i})$ are Einstein spaces of dimension $k-1$ and $n-(k-1)$, respectively.  But the Ricci tensor \eqref{eq:riccitensorindep} is not block diagonal, so $g$ is not isometric to a product metric and thus $(\mathcal{I}, [g])$ admits no $k$-dimensional non-degenerate $\text{Hol}_{x}(\mathcal{I}, [g])$-invariant subspaces for $2 \leq k \leq 4$.  It follow by definition, then, that $(\mathcal{I}, [g])$ is conformally indecomposable. Now, evaluating the expression for the cone metric \eqref{eq:conemetric}, the signature of $g_{C} = (1,5)$ since $\mathsf{scal}(g) < 0$.  $(\mathcal{I}, g)$ is simply connected and thus orientable, and since $(\mathcal{I}, g)$ and $\mathbb{R}^{+}$ are orientable manifolds, the metric cone $C((\mathcal{I}, g)) = (\mathcal{I}, g) \times \mathbb{R}^{+}$ is likewise orientable.  By Berger's classification \cite{berger}, metric holonomies of orientable semi-Riemannian manifolds of signature $(p,q)$ are contained in $SO(p,q)$, and so by \eqref{eq:confholmathcesconehol}, $\text{Hol}(\mathcal{I}, [g]) \subset SO(1,5)$.  Taking all these facts together, we have shown that $(\mathcal{I}, [g])$ is a $4$-dimensional simply connected, conformally indecomposable, scalar-negative Riemannian Einstein space with $\text{Hol}(\mathcal{I}, [g]) \subset SO(1,5)$, so Armstrong's classification of Riemannian conformal holonomy \ref{armstrong} applies and $\text{Hol}(\mathcal{I}, [g]) = SO^{0}(1,4)$ as claimed.
\end{proof}



\begin{thebibliography}{99}

\bibitem{alt}
J. Alt, ``Fefferman Constructions in Conformal Holonomy'' PhD Thesis, available at: \url{http://www.raumzeitmaterie.de/docs/dissertation_jesse_alt_2008.pdf} (2008)

\bibitem{armstrong}
S. Armstrong, ``Definite signature conformal holonomy: A complete classification'' J. Geom. Phys. \textbf{57} 10, 2024–2048 (2007)

\bibitem{arwini}
K. Arwini and C. T. J Dodson, \emph{Information Geometry: Near Randomness and Near Independence} (Springer, 2008)

\bibitem{bailey}
T. N. Bailey, M. G. Eastwood and A. R. Gover, ``Thomas's Structure Bundle for Conformal, Projective and Related Structures'' Rocky Mountain J. Math. \textbf{24}, 4 1191-1217 (1994)

\bibitem{baumreview}
H. Baum, ``Holonomy Groups of Lorentzian Manifolds: A Status Report'' in \emph{Global Differential Geometry} (Springer Proceedings in Mathematics Volume 17, 2012)

\bibitem{baumjuhl}
H. Baum and A. Juhl, \emph{Conformal Differential Geometry: $Q$ Curvature and Conformal Holonomy} (Birkh\"{a}user Verlag AG, 2010)

\bibitem{berger}
M. Berger, ``Sur les groupes d'holonomie homog\`{e}ne des vari\'{e}t\'{e}s a connexion affine et des vari\'{e}t\'{e}s Riemanniennes'' Bull. Soc. Math. France \textbf{83}, 279-330 (1955)

\bibitem{besse}
A. L. Besse, \emph{Einstein Manifolds} (Springer, New York, 1987)

\bibitem{brinkmann}
H. W. Brinkmann, ``Einstein spaces which are mapped conformally on each other'' Math. Ann. \textbf{94}, 1 119-145 (1925)

\bibitem{discalaolmos}
A. J. di Scala and C. Olmos, ``The geometry of homogeneous submanifolds of hyperbolic space'' Math. Z. \textbf{237}, 199–209 (2001)

\bibitem{eastwood}
M. Eastwood, ``Notes on conformal differential geometry'' Rend. Circ. Mat. Palermo (2) Suppl \textbf{43}, 57-76 (1996)

\bibitem{galaevleistner}
A. Galaev and T. Leistner, ``Recent developments in pseudo-Riemannian holonomy theory'' in V. Cort\'{e}s, \emph{Handbook of pseudo-Riemannian geometry and Supersymmetry}, (European Mathematical Society, 2010)

\bibitem{goverholography}
A R Gover, ``\pe Holography for Forms via Conformal Geometry in the Bulk'' Mem. Amer. Math. Soc. (to appear), e-print: arXiv:1205.3489v2 [math.DG]

\bibitem{governurowski}
A. R. Gover and P. Nurowski, ``Obstructions to conformally Einstein metrics in $n$ dimensions'' J. Geom. Phys. \textbf{56},  3 450-484 (2006)

\bibitem{govershaukatwaldron2} 
A. R. Gover, A. Shaukat and A. Waldron, ``Tractors, Mass, and Weyl Invariance'' Nucl. Phys. B. \textbf{812}, 3 424-455 (2009)

\bibitem{govershaukatwaldron1}
A. R. Gover, A. Shaukat and A. Waldron, ``Weyl invariance and the origins of mass'' Phys. Lett. B. \textbf{675}, 1 93-97 (2009)

\bibitem{goverboundary} 
A. R. Gover and A. Waldron, ``Boundary calculus for conformally compact manifolds'' arXiv:1104.2991 [math.DG] (2012)

\bibitem{kath}
I. Kath, ``Classification results for pseudo-Riemannian symmetric spaces'' in V. Cort\'{e}s, \emph{Handbook of pseudo-Riemannian geometry and Supersymmetry}, (European Mathematical Society, 2010)

\bibitem{kuhnelrademacher}
W. K\"{u}hnel and H.-B. Rademacher, ``Conformal transformations of pseudo-Riemannian manifolds'' in D. V. Alekseevsky and H. Baum, \emph{Recent developments in pseudo-Riemannian geometry}, (European Mathematical Society, 2008)

\bibitem{kulkarni}
R. S. Kulkarni and U. Pinkall, \emph{Conformal Geometry} (Aspects of Mathematics, Vieweg+Teubner Verlag Wiesbaden, 1988) 

\bibitem{leitnerhab}
F. Leitner, ``Applications of Cartan and Tractor Calculus to Conformal and CR-Geometry''. Habilitation thesis,  vailable at: \url{http://elib.uni-stuttgart.de/opus/volltexte/2009/3922/pdf/HabilKopf.pdf}

\bibitem{leitner}
F. Leitner, ``Aspects of conformal holonomy'' in W. Ebeling, K. Hulek and K. Smocyzk \emph{Complex and Differential Geometry: Conference held at Leibniz Universit\"{a}t Hannover, September 14-18, 2009} (Springer-Verlag, 2011)

\bibitem{li}
D. Li, H. Sun, C. Tao and L. Jiu, ``Riemannian Holonomy Groups of Statistical Manifolds'', e-print arXiv:1401.5706 [math.DG]



\end{thebibliography}
\end{document}